\documentclass{article}
\usepackage[utf8]{inputenc}
\usepackage[english]{babel}

\usepackage{amsmath}
\usepackage{amssymb}
\usepackage{ amssymb }
\usepackage{amsthm}
\usepackage{changepage}
\usepackage{ dsfont }
\usepackage{ hyperref}
\usepackage{cleveref}
\usepackage{graphicx}
\graphicspath{ {c:/user/photos/} }
\usepackage{tikz,tkz-graph}

\def\ws{\boxplus}
\def\ns{\oplus}
\def\wsns{\boxplus\oplus}

\def\nimber{\mathcal{N}}

\theoremstyle{plain}
\newtheorem{theorem}{Theorem}
\newtheorem{lemma}{Lemma}

\newtheorem{conjecture}{Conjecture}
\newtheorem{corollary}{Corollary}
\newtheorem{question}{Question}
\newtheorem{definition}{Definition}

\title{The Sprague-Grundy function for some selective compound games}
\author{Calvin Beideman, Matthew Bowen, Necati Alp Müyesser}
\date{February 15, 2018}

\begin{document}

\maketitle
\begin{abstract}
    We analyze the Sprague-Grundy functions for a class of almost disjoint selective compound games played on Nim heaps.  Surprisingly, we find that these functions behave chaotically for smaller Sprague-Grundy values of each component game yet predictably when any one heap is sufficiently large.
\end{abstract}

\section{Introduction}
In this paper we concern ourselves with two-player impartial combinatorial games under normal play. Thus the games we consider are perfect-information, both players are allowed the same set of moves given the same configuration of the game board, and the game eventually terminates. The player whose move terminates the game wins. From now on, we simply refer to these as games. For an overview of such games see \cite{winningways}.
\par Games can be modelled by a directed graph $(V, E)$ which we call the game tree. $V$ denotes the set of game states, whereas an edge $(v_1, v_2)$ denotes the existence of a move from state $v_1$ to state $v_2$. The leafs of the tree are then the terminal positions. It follows by easy induction on the game tree that from every position, either $P1$ or $P2$ has a winning strategy. Given a game $G=(V,E)$, the Sprague-Grundy (SG) function $\nimber:V\rightarrow\mathbb{N}$ generalizes this partition. From $v\in V$, the player who is about to play has a winning strategy if and only if $\nimber(v)\neq 0$. We usually call the Sprague-Grundy value of a game-state $v$ its \textit{nimber}.
\par A lot of our results build on the following recursive definition of the Sprague-Grundy function:
\begin{definition}
  Let $G=(V, E)$ be a game. If $v\in V$ is terminal, $\nimber(v)=0$. Otherwise, $\nimber(v)=mex\,\{\,\nimber(v')\,|\,(v,v')\in E\}$, where $mex$ denotes the minimum excluded value of a set in $\mathbb{N}$.
\end{definition}
\vspace{5mm}

In \textit{On Numbers and Games} \cite{numbersandgames} Conway suggests three potential rules for moving in compound games where games $G$ and $H$ are played simultaneously:
\begin{itemize}
    \item The disjunctive compound, denoted $G\ns{H}$.  Here players make a legal move in either $G$ or $H$ on their turn.
    \item The selective compound, denoted $G\ws{H}$.  Here on a player's turn they select either $G$, $H$, or both and makes legal moves in the ones selected.
    \item The conjunctive compound, where players always make legal moves in both component games.
\end{itemize}

\par Given enough information about each of the component games the Sprague-Grundy theorem makes it easy to determine the $SG$-function $\nimber{}$ for the disjunctive sum of two games: $\nimber{(G\ns{H})}=\nimber{(G)}\ns \nimber{(H)}$, where the second $\ns$ denotes the bitwise xor operation on $\nimber{(G)}$ and $\nimber{(H)}$. As an example, by $*k$ we denote the game of a Nim pile with $k$ stones. A valid move is to remove an arbitrary amount of stones from the pile. Then clearly by Definition 1, $\nimber(*k)=k$. One pile Nim is not a very interesting game; however, $(*k)\ns(*l)\ns(*m)$ can be easily navigated by computing nimbers, even though there isn't an intuitive winning strategy always.

\par The $SG$-function of selective compound games, however, is not characterized by the nimbers of its component games: for example $\nimber{(*1\ws{*0})}=1\neq \nimber{(*1\ws{(*1\ns *1)})}=3$ even though the nimbers of the component games agree.  In fact, even for games as simple as these determining the $SG$-function can be rather complicated.   In 2015 Boros et. al. \cite{exconim} gave a partial analysis of $\nimber{(*a\ws{(*b\ns *c)})}$ and noted that this function behaves rather chaotically.  We continue this analysis by proving some of the conjectures presented in \cite{exconim} as well as extending results to the game $\nimber{(*x_1\ws{(*x_2\ns...\ns *x_n}))}$.  We call this game \textit{Auxiliary Nim} and more generally, for a given game $G$ we call the game $*k\ws G$ \textit{Auxiliary G}.
\par A lower bound and an upper bound can easily by derived for the nimber of a Auxiliary Nim game. We show the following bounds in Corollary \ref{thm:easyBounds}:
$$x_1+(x_2\ns x_3 \ns \cdots \ns x_n)\leq\nimber(x_1,x_2,x_3, \cdots, x_n)\leq x_1+x_2+x_3+\cdots+x_n$$
\par Two of our main results characterize when these extreme points are realized.

\vspace{2mm}

\textbf{Question 1: }Under which circumstances $\nimber{(*x_1\ws{(*x_2\ns...\ns *x_n})}=x_1$, the lowest achievable value by Corollary \ref{thm:easyBounds}?\vspace{1mm}
\par Theorem \ref{thm:firstTheorem} completely answers this question:

\begin{theorem}\label{thm:firstTheorem}
$\nimber(*x_1\ws (*x_2\ns \cdots \ns *x_n))=x_1 \Leftrightarrow$ $ (*x_2\ns \cdots \ns *x_n)=0$ and $ 2^{\lfloor \log_2 x_1\rfloor + 1}$ divides all of $x_2,x_3, \cdots, x_n$.
\end{theorem}

\textbf{Question 2: } Under which circumstances is the upper bound from Corollary \ref{thm:easyBounds} realized? \vspace{1mm}
\par The answer turns out to be that the upper bound is realized when $x_1$ is sufficiently large compared to the other $x_i$s. We first define $A(x_2,...,x_n)$ to be the least value of $x_1$ such that $\forall a\geq x_1$, $\nimber{(*a\ws{(*x_2\ns...\ns *x_n})}=a+x_2+...+x_n$.
\begin{theorem}\label{thm:bigA}
Let $(x_1,x_2 \cdots, x_n)$ be an Auxiliary-Nim game with $n$-many piles. Then, $A(x_2, \cdots, x_n)$ is well-defined. Furthermore, $A(x_2, \cdots, x_n)$ grows quadratically with respect to the sum $x_2+ \cdots + x_n$.
\end{theorem}
Further, in the special case of $n=3$, we prove a linear upper bound. In Lemma \ref{lemma:a0UpperBound}, we show that $$A(b,c) \leq min(\sim b,\sim c)+1$$ where $\sim x$ denotes the bitwise complement. We also provide some sufficient conditions for this upper bound to be realized. The Analysis of the $n=3$ case brings us to the next question.
\vspace{2mm}

\textbf{Question 3: } Can we come up with a closed-form, non-recursive way to describe the behaviour of $\nimber{(*a\ws(*b\ns *c)}$, the Auxiliary Nim game with only $3$ piles?
\vspace{1mm}
\par Question 3 is still open. We to show a linear upper bound on $A(b,c)$, and partially resolved the cases where $b$ and $c$ are sufficiently close to a power of $2$. In particular, we show the following:

\begin{theorem}\label{thm:same-char}
Suppose $b=2^i+k$ and $c=2^i+l$ with $k<l<2^i$.  Then
\[ \nimber{(a,b,c)}= \begin{cases}
      a+b+c & a \geq 2^i-l\\
      2a+c+k+l  & 2^i-k-l\leq a<2^i-l\,;\, l\leq2^{i-1}\\
      \geq\nimber{(a,k,l)} & l>2^{i-1}\,;\,\nimber{(a,k,l)}\geq 2^i\\
      \nimber{(a,k,l)} &  \nimber{{(a,k,l)}<2^i}
   \end{cases}
\]
\end{theorem}

This recursive structure causes the $SG$ function to become rather complicated, even in simple circumstances.  For a qualitative view of this complexity, see Figure \ref{fig:fig}.

\vspace{1 mm}
We also get closer to a complete characterization of $\nimber(*1\ws (*b\ns*c))$:

\begin{theorem}\label{thm:oddb}
For $b$ odd, if $c\geq 2^{2\lfloor \log_2 b \rfloor+1}-2^{\lfloor \log_2 b \rfloor+2}-1$ then $\nimber{(*1\ws (*b\ns*c))}=1+b+c$.
\end{theorem}
Therefore, there are at least some cases where the $SG$-function of this game is well-behaved. But outside the domain of the assumptions of the previous theorems, even in the analysis of the simplest possible Auxiliary Game, the function $\nimber{(*1\ws (*b\ns*c))}$ seems to result in combinatorial chaos.

\begin{figure}
    \centering
    \includegraphics{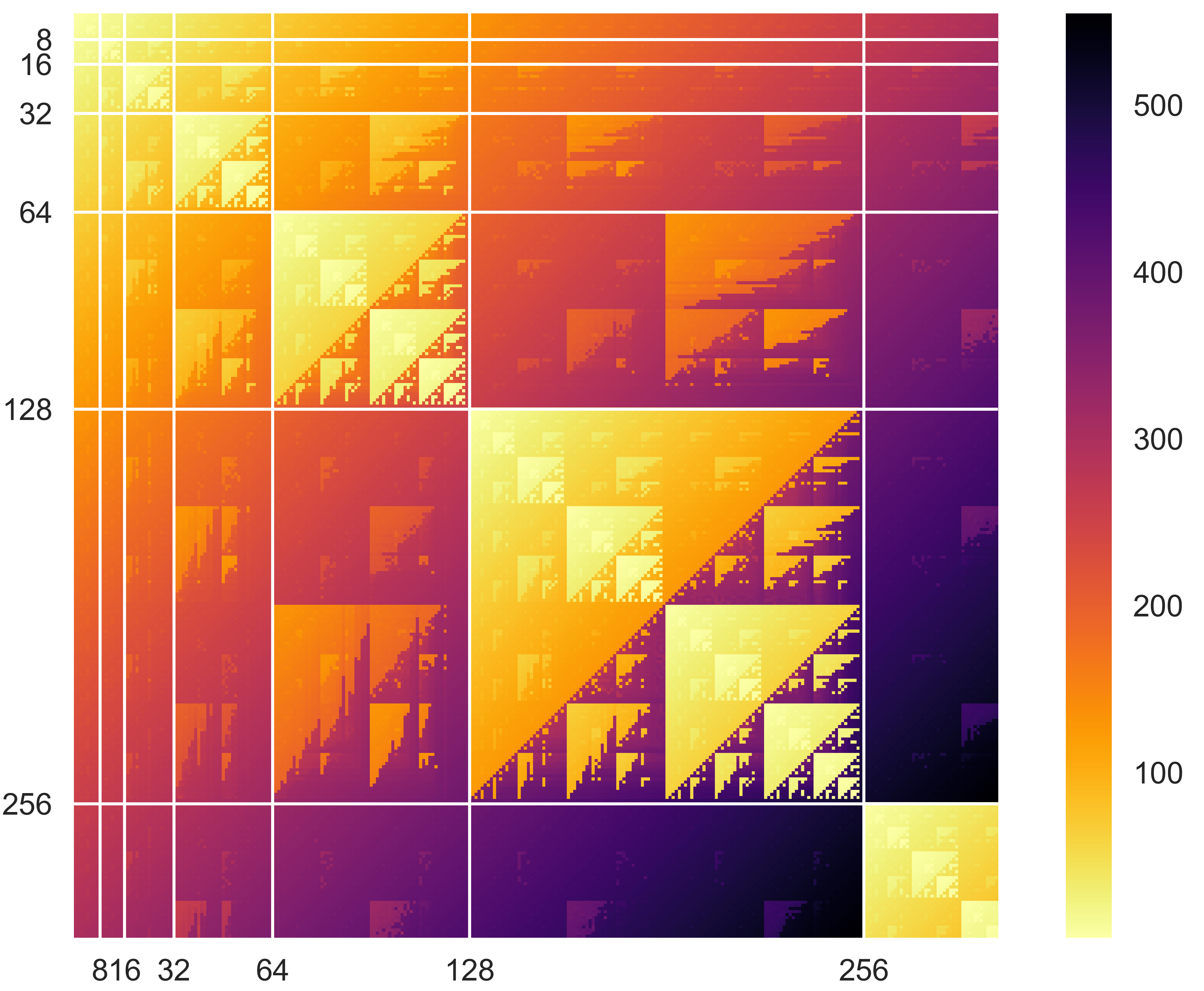}
    \caption{A heat-map for the Sprague-Grundy values (nimbers) for the game $(*1)\ws (*x\ns *y)$.  The behavior of the blocks of size $2^n$ along the diagonal are characterized by Theorem \ref{thm:same-char}. The structure of the fixed blocks ``decay'' as they are translated to the right/down.  This is partially explained by Theorem \ref{thm:oddb}.}
    \label{fig:fig}
\end{figure}
\begin{figure}
    \centering
    \includegraphics{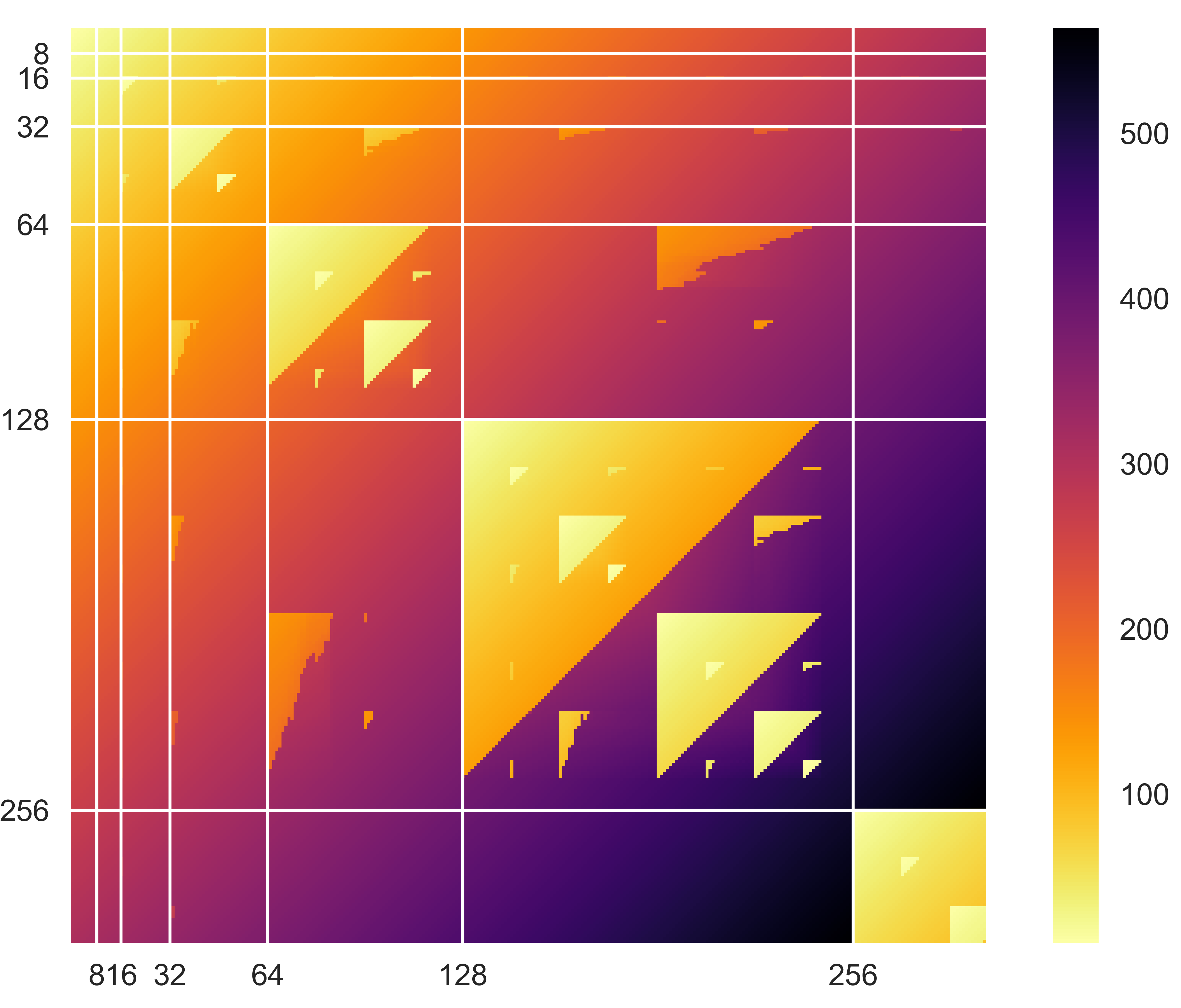}
    \caption{A heat-map for the Sprague-Grundy values (nimbers) for the game $(*8)\ws (*x\ns *y)$. Notice that the auxiliary pile size is larger compared to the game in Figure \ref{fig:fig}, and the heat-map looks more ``orderly''. This is partially explained by Theorem \ref{thm:bigA}, in particular, by the fact that $A(b,c)\leq min(\sim b,\sim c)+1$ (Lemma \ref{lemma:a0UpperBound}).
    Nimbers achieve the lower bound (in this case, $8$) only along the diagonal when $b=c$ is a multiple of $16$, as shown by Theorem \ref{thm:firstTheorem}.}
    \label{fig:fig2}
\end{figure}

\section{Results}
From now on, we will refer to the game $(*a)\ws(*b\ns *c)$ simply as $(a,b,c)$, and similarly $(*x_1)\ws(*x_2\ns *x_3 \ns \cdots \ns *x_n)$ as $(x_1,x_2,x_3, \cdots, x_n)$. Also, $\nimber(a,b,c)$ denotes the Sprague-Grundy value of the game $(a,b,c)$. Finally, we use $(a,b,c)\rightarrow N$ to state that the game $(a,b,c)$ can reach a game with nimber $N$ through some legal move. Similarly, $(a,b,c)\nrightarrow N$ means that the game $(a,b,c)$ cannot reach a game with nimber $N$. Observe that $(a,b,c)\rightarrow N$ implies $\nimber(a,b,c)\neq N$.

We begin with some preliminary results:

\begin{lemma} \label{thm:NimberIncreasesWithA}
$\nimber(a,b,c)>\nimber(a-1,b,c)$. $\forall a\in\mathds{N}$
\end{lemma}
\begin{proof}
We see that if $(a-1,b,c)\rightarrow N$, $(a,b,c)\rightarrow N$, by first setting $a$ to $a-1$, and replicating the remaining move. Moreover, $(a,b,c)\rightarrow \nimber(a-1,b,c)$, thus $\nimber(a,b,c)>\nimber(a-1,b,c)$ as desired.
\end{proof}
\begin{corollary}
$\nimber(x_1,x_2,x_3, \cdots, x_n)>\nimber(x_1-1,x_2,x_3, \cdots, x_n)$. $\forall x_1\in\mathds{N}$
\end{corollary}
\begin{lemma} \label{thm:CrudeBounds}
$a+(b\ns c)\leq\nimber(a,b,c)\leq a+b+c$
\end{lemma}
\begin{proof}
The upper bound is trivial, since $a+b+c$ is the depth of the game $(a,b,c)$. We prove the lower-bound by induction on $a$. Let $b,c$ be arbitrary and fixed. For the base case, clearly $\nimber(0,b,c)=b\ns c \geq 0 + (b\ns c)$. Assuming that the bounds holds for lower values of a, we get $\nimber(a-1,b,c)\geq a-1 + (b\ns c)$ by hypothesis. By Lemma 1, we have that $\nimber(a,b,c)\geq a + (b\ns c)$ as desired.
\end{proof}
\begin{corollary}\label{thm:easyBounds}
$x_1+(x_2\ns x_3 \ns \cdots \ns x_n)\leq\nimber(x_1,x_2,x_3, \cdots, x_n)\leq x_1+x_2+x_3+\cdots+x_n$
\end{corollary}
\begin{proof}
We see that the lower bound in Lemma \ref{thm:easyBounds} immediately generalizes to the case where we have arbitrary number of piles, as moves on the right hand side, as well as in the auxiliary pile can be replicated in a similar fashion. The upper bound also does, as the depth of the game still is a trivial upper bound on the nimber of the game.
\end{proof}

Now, we begin by providing a necessary and a sufficient condition for $\nimber(a,b,c)$ to simply evaluate to $a$, and then we generalize this to a complete proof of Theorem $\ref{thm:firstTheorem}$.

\begin{lemma}\label{thm:specialCase}
$\nimber(a,b,c)=a \Leftrightarrow$ $\exists k\in\mathds{N}$. $ b=c=k\cdot 2^{\lfloor \log_2 a\rfloor + 1}$.
\end{lemma}
The theorem claims that $\nimber(a,b,c)=a$ if and only if $b=c$ is a multiple of a power of 2 strictly greater than $a$. Note this is just a special case of Theorem \ref{thm:firstTheorem}. The proof of the special case is easier to formalize, and generalizes painlessly, so we provide a proof.
\begin{proof}
We begin by the $(\Leftarrow)$ direction. If $k=0$, the statement is trivial. Therefore, let $b=c$ be a multiple of a power of 2 strictly greater than $a$. Thus in the binary representation, $b$ has as at least as many $0$s as the number of bits in $a$. It suffices to show $(a,b,b)\nrightarrow a$ to conclude $\nimber(a,b,b)=a$, since we have $\nimber(a,b,b)\geq a+(b\ns b) = a$ by Lemma 2.

\begin{equation}
\frac{
    \begin{array}[b]{r}
      \left( 1 \cdots x \cdots y \right)\\
      \left( 1\cdots 1\cdots 0\cdots0 \cdots0 \right) \\
      \wsns \left( 1\cdots1\cdots0\cdots0\cdots0 \right)
    \end{array}
  }{
    \left( \nimber(a,b,c)\right)
  }
\end{equation}
From diagram 1, we observe that any move that decreases $b$ to $b'$ ensures that $b\ns b' > a$, since a decrease in $b$ implies flipping a $1$ bit to the left of the leftmost bit in $a$, therefore in the xor operation, the bit from the other $b$ will fall down, to the left of $a$. So by the lower bound in Lemma 2, any such move will never obtain a nimber equal to $a$, since $\nimber(a,b,b')>a$. \\[0.04in]
We still need to show that $(a,b,b)\nrightarrow a$, but we are now only concerned with moves only decrease the first pile. For this case, we induct on $a$. Since we assume we can decrease $a$, $a$ has to be non-zero. When $a=1$, decreasing $a$ is equivalent to removing the first pile, thereby resulting in the game $(b\ns b)$ with nimber $0\neq a$. In the inductive step, we assume that we decrease the size of the first pile by $k$, yielding game $(a-k,b,b)$. By assumption, we have $b=k\cdot 2^{\lfloor \log_2 a\rfloor + 1}$. But by a decrease in $a$, we cannot change the fact that b is still a multiple of a power of two strictly greater than $a$. Hence, $b=l\cdot 2^{\lfloor \log_2 a-k\rfloor + 1}$, and the inductive hypothesis applies to show $\nimber(a-k, b, b)=a-k\neq a$. This concludes the induction, and the $(\Leftarrow)$ direction of the Theorem. \\[0.06in]
We will show the $(\Rightarrow)$ direction by contrapositive. When $b$ and $c$ are not the multiple of the power of two that we require, we want to show $\nimber(a,b,c)\neq a$ Suppose first that $b\neq c$. Then $b\ns c \neq 0$, and by the bound from Lemma 2, we see that $\nimber(a,b,c)>a$, so we are done. \\[0.04in]
Now, suppose $b=c$, but $b$ is not a multiple of a power of two strictly greater than a. We will show $(a,b,b)\rightarrow a$ by induction on a. \\ [0.04in]
In the base case, $a=1$. Then,
\begin{align*}
2^{log_2 \lfloor a \rfloor +1}
&= 2^{log_2 \lfloor 1 \rfloor +1} \\
&= 2^{1} \\
\end{align*}
Therefore, we deduce $b\neq2k$ by assumption, i.e. $b$ is odd. We observe that $b\ns (b-1)=1$, as $b-1$ is simply $b$ with the right-most bit inverted, since b is odd. Thus, $(1, b, b)\rightarrow 1$, and we have a base case. \\[0.04in]
In the inductive step, we consider $(a,b,b)$. We assume $b$ is not a multiple of a power of 2 strictly greater than $a$.\\[0.03in]
\textbf{Case 1. } $b$ also is not a multiple of a power of 2 strictly greater than $a-1$. In this case, the hypothesis applies to the game $(a-1, b, b)$, to show $(a-1, b, b)\rightarrow a-1$. From the bounds in Lemma 1 and 2, it is evident that:
\begin{align*}
    \nimber(a,b,b)&>\nimber(a-1,b,b)\\
                 &\geq (a-1) + 1 \\
                 &= a
\end{align*}
and thus we are done. \\[0.03in]
\textbf{Case 2. } $b$ is a multiple of a power of 2 strictly greater than $a-1$, but not a multiple of a power of two strictly greater than $a$. We conclude that in this case, $a=2^k$ for some $k$, as that is the only way the power of $2$ strictly greater than $a-1$ would not also be strictly greater than $a$. \\[0.03in]
We also see that $b$ is a multiple of $a$ in this case, and we thus see that in the base $2$ representation, $b$ has to have a $1$ bit at the $k^{th}$ index and thus contain a ``copy'' of $a$, as otherwise, $b$ would simply be the multiple of $2^{k+1}$, contradicting our assumption. (This is equivalent to stating that $b$ is an odd multiple of $a$.) Thus we have, $(a,b,b)\rightarrow b \ns (b-a)=a$, as desired. We show this bit argument in the diagram below.

\begin{equation}
\frac{
    \begin{array}[b]{r}
      \left( 1 \cdots 0 \cdots 0 \right)\\
      \left( 1\cdots 1\cdots 1\cdots0 \cdots0 \right) \\
      \wsns \left( 1\cdots1\cdots1\cdots0\cdots0 \right)
    \end{array}
  }{
    \left( \nimber(a,b,b)\right)
  }
\end{equation}
The above diagram gets converted to the below diagram, with the move that eliminates the first pile, and decreases $a$ from the second pile. Note that in the case when $a=b$, this procedure simply amounts to removing piles 1 and 2.
\begin{equation}
\frac{
    \begin{array}[b]{r}
      \left( 1\cdots 1\cdots 0\cdots0 \cdots0 \right) \\
      \ns \left( 1\cdots1\cdots1\cdots0\cdots0 \right)
    \end{array}
  }{
     \,\,\,\,\,\,\,\,\,\,\,\,\,\,\,\,\,\,\,\,\,\,\,\,\,\,\,\,\,\,\left(1 \cdots 0 \cdots 0 \right)
  }
\end{equation}

\end{proof}
We are now ready to prove Theorem \ref{thm:firstTheorem} in full generality. For convenience, we restate it below:\vspace{2mm}
\par \textbf{Theorem 1. } $\nimber(x_1,x_2,x_3, \cdots, x_n)=x_1 \Leftrightarrow$ $ (x_2\ns x_3\ns \cdots \ns x_n)=0$ and $ 2^{\lfloor \log_2 x_1\rfloor + 1}$ divides all of $(x_2,x_3, \cdots, x_n)$.\vspace{2mm}
\par The Theorem strengthens Lemma \ref{thm:easyBounds} to characterize all the Auxiliary Nim games where the nimbers are equivalent to the size of the first pile. Note that unlike in the statement of Lemma \ref{thm:easyBounds}, we do not and cannot mandate that all the values $(x_2,x_3, \cdots, x_n)$ are equivalent. We merely require that all the values xor to $0$ (in the $3$ pile game, this is equivalent to saying $b=c$).
\begin{proof}
For the $(\Leftarrow)$ direction, we have that all of $(x_2,x_3, \cdots, x_n)$ xor to $0$ and each have as many $0$s as the number of bits of $x_1$. Thus for any move that is not solely a decrease in the $(*x_1)$ pile, a decrease in a pile $(*x_i)$ to $(*x'_i)$ ensures that $(x_2\ns \cdots x'_i \ns \cdots \ns x_n)>x_1$, as a bit falls down to the left of $x_1$, and the xor was $0$ before the move, by assumption. For moves that decrease only the $(*x_1)$ pile, we can induct on the value of $(*x_1)$ to show that all such decreases will yield a nimber of $x_1-d$, where $d$ is the decrease. This step is identical in the proof for Theorem 1. \\[0.07in]
We now show the $(\Rightarrow)$ direction, again by contrapositive. By the lower bound in Lemma \ref{thm:easyBounds}, it follows immediately that $(x_2\ns x_3 \ns \cdots \ns x_n)=0$, as otherwise, $\nimber(x_1,x_2,x_3, \cdots, x_n)>x_1$. So we assume that there exists $x_i\in(x_2,x_3,\cdots, x_n)$ such that $2^{\lfloor \log_2 x_1\rfloor + 1}$ does not divide $x_i$. We again induct on the value of $x_1$. In the base case when $x_1=1$, we conclude $x_i$ is odd, thus the move that sets $x'_i=x_i-1$ will yield nimber $1$, just like in the previous proof, contradicting $\nimber(a,b,c, \cdots, z)=a$. We again separate our inductive step into two cases. If our inductive hypothesis applies to the same game with $x'_1=x_1-1$, we are done. Otherwise, $x_1=2^m$ for some $m$ power of $2$, and we have a $x_i$ such that $x_i$ is a multiple of $2^m$, but not $2^{m+1}$, and thus $x_i$ contains a ``copy'' of the bits of $x_1$, i.e. $x_i$ has a $1$ bit at the $m^{th}$ index. We set $x'_i=x_i-x_1$ to yield a game with nimber $x_1$, thus showing that the nimber of the original game could not have been $x_1$, concluding the proof.
\end{proof}

We note that despite the fact that the SG-values of $(*a)\ws(*b\ns*c)$ is complex when the value of $a$ is low, the SG value merely equals $a+b+c$, i.e. the upper bound, when $a$ is large enough. In the following section of the paper, we formalize this notion, and give some characterizations of the cases for when $\nimber(a,b,c)=a+b+c$.
\begin{definition}
For any $b,c \in \mathbb{N}$, we define $A(b,c)$ to be the minimum $a \in \mathbb{N}$ st. $\nimber(a,b,c) = a + b + c$
\end{definition}

Note that it is not necessarily clear from the definition that $A$ is even well defined. Soon, however, will prove this, by establishing an upper bound on $A(b,c)$.

\begin{lemma} \label{thm:NimberIsAlwaysSum}
If $A(b,c)$ is defined, for every $a > A(b,c)$, $\nimber(a,b,c) = a + b + c$.
\end{lemma}

\begin{proof}
This follows immediately from the lower bound provided by \Cref{thm:NimberIncreasesWithA} and the upper bound provided by \Cref{thm:CrudeBounds}.
\end{proof}

\begin{definition}
Let $n_i$ be the value of the $i^{th}$ digit of n in its binary representation, indexing from zero and the right.
We call $n$ a $gap$ in $a \ns b$ if $n=a \ns b$ or if at the leftmost index $i$ in the binary representation of n where n differs from $a \ns b$ we have $n_i=1$ and $a_i=b_i=0$.
\end{definition}
Note that if $n \geq 2^{\lceil log_2(max(a,b)) \rceil}$ then n is always a gap.

\begin{lemma} \label{lemma:nongapsareattainable}
If n is not a gap in $b \ns c$ then $(0,b,c) \rightarrow n$
\end{lemma}
\begin{proof}
Consider the left most bit $i$ in n that differs from $b \ns c$.  Note that to get from $(0,b,c)$ to $n$ we will never have to alter bits to the left of index $i$. There are two cases.

\textbf{Case 1: }$n_i=1$.  Then $b_i=c_i=1$.  Let $b'=c \ns n$.  Clearly $b' \ns c=n$.  Further, $b'_j=b_j$ $\forall j>i$ as $b_j \ns c_j=n_j$ by the definition of a gap, and $b_i=1 > b'_i=0$, so $b>b'$.  Thus the move from $(0,b,c)$ to $(0,b',c)$ is valid, so $(0,b,c) \rightarrow n$ as desired.

\textbf{Case 2: } $n_i=0$. WLOG let $b_i=1$ and $c_i=0$.  Letting $b'=c \ns n$ as before, by the same logic we have the move from $(0,b,c)$ to $(0,b',c)$ is valid and $\nimber{(0,b',c)}=n$ as desired.
\end{proof}

We say $n$ is the $j^{th}$ gap in $a \ns b$ if $n$ is a gap and there are precisely $j-1$ gaps $n'$ such that $n'<n$. Note that $a\ns b$ will always be the first gap in $a\ns b$.
\begin{lemma} \label{lemma:gapsarelowerbounds}
Let n be the $j^{th}$ gap in $b \ns c$. Then $\nimber{(j-1,b,c)} \geq n$.
\end{lemma}
\begin{proof}
Induction on $j$.  If $j=1$, then there are no gaps in $b \ns c$ less than $n$, so by \cref{lemma:nongapsareattainable}, $\nimber{(0,b,c)} \geq n$.  Now suppose $j>1$ and let $n'$ be the $(j-1)^{st}$ gap.  Then by the inductive hypothesis, $\nimber{(j-2,b,c)} \geq n'$, so $(j-1,b,c) \rightarrow i$. $\forall i \in[n']$, by reducing $j-1$ to $j-2$ and replicating the rest of the move.  But as there are no gaps between $n$ and $n'$, $(0,b,c) \rightarrow i$. $\forall i \in[n'+1,n-1]$.  Thus $\nimber{(j-1,b,c)} \geq n$ as desired.
\end{proof}
\begin{lemma} \label{lemma:a0UpperBound}
For any $b,c \in \mathbb{N}$, $A(b,c)$ is defined, and $A(b,c) \leq min(\sim b,\sim c)+1$, where $\sim x$ denotes the bitwise complement.
\end{lemma}
This theorem establishes a linear upper bound on $A(b,c)$ for any $b$ and $c$, thereby proving that $A(b,c)$ is well-defined for arbitrary values.  Further, it proves Conjecture $2$ and a special case of conjecture $3$ posed in \cite{exconim}.
\begin{proof}
Let $a=min(\sim b,\sim c)+1$.  It suffices to show $\forall n<a+b+c .$ $(a,b,c) \rightarrow n$.  Then by the upper bound from Lemma 2, $\nimber{(a,b,c)=a+b+c}$, so $A(b,c) \leq min(\sim b,\sim c)+1$ as desired.
Proceed by induction on $b+c$.
\\[0.04in]
For the base case, if $b+c=0$, $b=c=0$ and the claim is trivially true.  Now suppose $b+c>0$. We will case on whether $\sim b$ or $\sim c$ have the greater value, and assume WLOG that $\sim b \leq \sim c$. So $a = \sim b + 1$.
\par We will first show that $(a,b,c) \rightarrow n$  $\forall n \in[a+b, a+b+c-1]$. Let $i \in[c]$.  Observe that $min(\sim b, \sim c) \geq min(\sim b, \sim (c-i))$.  So by the induction hypothesis $a \geq A(b,c-i)$   and thus $\nimber{(a,b,c-i)}=a+b+c-i$.

We will now cover the rest of the range, so we want to show $(a,b,c) \rightarrow n$  $\forall n \in[0, a+b-1]$.  From Lemma 5, we have that if there are $n$ gaps in $b \ns c$ less than or equal to $b$ $+ \sim b=a+b-1$ then $\nimber{(n,b,c)} \geq n'>b$ $+\sim b$ where $n'$ is the $(n+1)^{st}$ gap.  So it suffices to show that there are at most $\sim b+1=a$ gaps less than $a+b$.  But by the definition of gaps, the number of gaps less than $a+b$ is maximized if whenever $b_i=0$ it is also the case that $c_i=0$ for $i \leq log_2(b)$. If this is the case there are precisely $2^i$ gaps for each $i \leq log_2(b)$  such that $b_i=0$ and one gap to account for $b \ns c$.  Summing over all of these gaps, there are $a=\sim b + 1$ in total and the proof is complete.
\end{proof}
There are indeed non-trivial instances where the upper bound provided by \Cref{lemma:a0UpperBound} is strict, as we will show shortly. However, it is natural to suspect from the proof of the Theorem that the actual number of gaps less than $a+b$ is a suitable candidate for a better upper bound (we had assumed that the number of gaps is as large as it possibly can be in the proof). We now prove an extension to \Cref{lemma:a0UpperBound} for when this actually is the case.
\begin{lemma}\label{lemma:NumberOfGapsIsAnUpperBound}
Let $b=2^i+k$ and $c=2^j+l$ where $k<2^i$ and $l<2^j$ and $j>i$. Also assume whenever $b$ has a $1$ bit at the $n^{th}$ index of its binary representation, so does $c$. Then, $A(b,c)$ is bounded above by the number of gaps in $b\ns c$ less than $c$ $|$ $(b+\sim b)$, where $|$ is the bitwise $or$ operator.
\end{lemma}
\begin{proof}
Proof is by induction on $b$.\\
For the base, note $b=1$ and the claim holds for any valid choice of $c$ by \Cref{lemma:a0UpperBound}.
Now let $b$ be given, $c$ satisfying the conditions of the claim, and $a$ the number of gaps in $b\ns c$ less than $c$ $|$ $(b+\sim b)$. We will show $\nimber(a,b,c)=a+b+c$. \par We begin by noting that when we decrease $b$ to $b'$, the number of gaps less than $c$ $|$ $(b+\sim b)$ cannot increase. This is because by making a decrease in $b$, we cannot create a new index $n$ where $b$ and $c$ have both $0$ bits that did not exist originally, by the assumption. Therefore, $(a,b,c) \rightarrow n.$ $\forall n\in[a+c,a+b+c-1]$ by reducing $b$ and applying the induction hypothesis.  For the rest of the range, note that $a+c=c$ $|$ $(b+\sim b)$. Values less than $a+c$ are either attainable by bit arguments by \Cref{lemma:nongapsareattainable} or they are one of the $a$ gaps less than in $b\ns c$. In that case by \Cref{lemma:gapsarelowerbounds} $(a,b,c) \rightarrow n.$ $\forall n\in[0,a+c]$. Thus $\nimber(a,b,c)=a+b+c$, and $A(b,c)\leq a$, as desired.
\end{proof}
\par Unfortunately, the upper bound shown in \Cref{lemma:a0UpperBound} does not generalize in the obvious sense to the game with arbitrary amount of piles. However, we can show that $A(x_1,x_2, \cdots, x_n)$ is well-defined, and is bounded above quadratically. This was the statement of Theorem \ref{thm:bigA}, which we reproduce below for convenience.\vspace{2mm}
\par \textbf{Theorem 2. } Let $(x_1,x_2 \cdots, x_n)$ be an Auxiliary-Nim game with $n$-many piles. Then, $A(x_2, \cdots, x_n)$ is well-defined. Furthermore, $A(x_2, \cdots, x_n)$ grows quadratically with respect to the sum $x_2+ \cdots + x_n$. \vspace{2mm}
\begin{proof}
Proof is by induction on $x_2+\cdots+x_n$. When the sum is $0$, $A(x_2, \cdots, x_n)$ is trivially $0$ also. Otherwise let the sum be any positive integer. We know that if we make a decrease in any of the piles $x_2$ through $x_n$, the resulting collection of piles have a well-defined $A$ value, by induction. We set: $$a^*=A(x_2 -1, x_3, \cdots, x_n) + x_2 + \cdots + x_n$$ Then $\nimber(a^*, x_2, \cdots, x_n)>a^*$ by the lower bound from Corollary 2. For the remaining nimbers, we can simply consider the move when we subtract $1$ from the first pile ($*x_2$), and the nimber of the resulting game will hit the upper bound as long as we don't subtract more than $ x_2 + \cdots + x_n$ from $a^*$. Luckily, we only need to remove up to this much to hit all the nimbers in the range $[a^*, a^* + x_2 + \cdots + x_n]$. This concludes the proof.
\end{proof}

We are now in a position to begin proving explicit characterizations of $A(b,c)$ in several cases. We will make use of the following lemma which lower bounds the size of $A(b,c)$. Afterwards, we will show that in some non-trivial instances, the lower bound matches the upper bound derived from \Cref{lemma:a0UpperBound}.
\begin{lemma} \label{lemma:recursiveLowerBound}
$A(b,c)\geq min(A(b-1,c), A(b,c-1))$
\end{lemma}
\begin{proof}
AFSOC, $a=A(b,c)<min(A(b-1,c), A(b,c-1))$ and consider $(a,b,c)$.  Then $(a,b,c)\rightarrow a+b+c-1$.  But as $\nimber{(a,b,c)}\leq a+b+c$ by the upper bound from Lemma 2, we can only reach this value reducing one of $a,b,c$ by exactly $1$.  But none of $\nimber{(a-1,b,c)}$, $\nimber{(a,b-1,c)}$, or $\nimber{(a,b,c-1)}$ can be $a+b+c-1$ by the definition of $A(b,c)$, and the assumption. This is a contradiction.
\end{proof}

Note that the proof for \Cref{lemma:recursiveLowerBound} generalizes similarly to give a lower bound for $A(x_1,x_2,...,x_n)$.

\begin{corollary}
$A(x_1, \cdots, x_n)\geq min(A(x_1-1, \cdots, x_n), \cdots ,A(x_1, \cdots, x_n-1))$
\end{corollary}

\Cref{lemma:recursiveLowerBound} also allows us to characterize $A(b,c)$ when b and c are sufficiently close, as will be explicitly stated in \Cref{lemma:a0bothPowerOfTwo}.

\begin{lemma}\label{lemma:a0bothPowerOfTwo}
$A(2^i+x,2^i+y)=2^i-max(x,y)$ for $0\leq x,y < 2^i$.
\end{lemma}
\begin{proof}
$2^i - max(x,y)=\sim(2^i + x)+1$ is precisely the upper bound given by \Cref{lemma:a0UpperBound}, so it suffices to show our lower bound derived from Lemma \ref{lemma:recursiveLowerBound} corresponds with this as well.  This is done by induction on $x+y$.

For the base cases, let $y=0$. Then by \Cref{thm:conjectureone}, $A(2^i + x, 2^i)=2^i-x$, as $x<2^i$.

Now suppose the claim holds for $x+y=n$ and consider the case where $x'+y'=n+1$. WLOG, we can consider the case where $x'=x+1$ and $y'=y$. We can also assume $x,y>0$ since the other cases are covered already, meaning we can safely assume $y-1\geq0$ and apply the inductive hypothesis. By \Cref{lemma:recursiveLowerBound}, we have that:
\begin{align*}
    A(2^i+x+1, 2^i+y)&\geq min(A(2^i+x, 2^i+y), A(2^i+x+1, 2^i+y-1)) \\
    &=min(2^i - max(x,y), 2^i-max(x+1, y-1)) &&\text{By IH} \\
    &=2^i - max(x+1,y)
\end{align*}
Thus the lower bound matches the upper bound by induction.
\end{proof}

\par With this we can give a characterization of $\nimber{(a,b,c)}$ when $\lfloor\log_2(b)\rfloor=\lfloor\log_2(c)\rfloor$. In order to do this, however, we will need a result from Boros et al., which we restate below for convenience:

\begin{lemma}\label{lemma:addingPowerOfTwo}
Suppose that $a,b,c,i \in \mathbb{N}$. If $\nimber(a,b,c) < 2^i$  then $\nimber(a,b+2^i,c+2^i) = \nimber(a,b,c)$.  On the other hand, if $\nimber(a,b,c) < 2^i$ then $\nimber(a,b+2^i,c+2^i) \geq \nimber(a,b,c)$

\end{lemma}

\begin{proof}
See the proof of Lemma 7 in \cite{exconim}.
\end{proof}
We are now ready to prove Theorem \ref{thm:same-char}, which we restate below.\vspace{2mm}\\
\textbf{Theorem 3. }
Suppose $b=2^i+k$ and $c=2^i+l$ with $k<l<2^i$.  Then
\[ \nimber{(a,b,c)}= \begin{cases}
      a+b+c & a \geq 2^i-l\\
      2a+c+k+l  & 2^i-k-l\leq a<2^i-l; l\leq2^{i-1}\\
      \geq\nimber{(a,k,l)} & l>2^{i-1};\nimber{(a,k,l)}\geq 2^i\\
      \nimber{(a,k,l)} &  \nimber{{(a,k,l)}<2^i}
   \end{cases}
\]

\begin{proof}
The first and last two cases are covered by \Cref{lemma:a0bothPowerOfTwo} and  \Cref{lemma:addingPowerOfTwo} respectively, so it suffices to show $\nimber{(a,b,c)}=2a+b+c-(2^i-l)$ whenever $0^i-k-l<a<2^i-l$ and $l\leq 2^{i-1}$.  This can be done via induction on $k+l$: for the base case when $k=0$ see \Cref{thm:conjectureone}.  Otherwise, suppose $k+l>0$ and that the result holds for all previous examples.  In general, we can cover all values in the range $[2^i-1,2^{i+1}-1]$ by bit arguments alone.  For $a=2^i-k-l$, values in the range $[2^{i+1},3*2^{i}-k-l-1]$ can be reached by moving to the positions $(a',2^i-l-1,2^i+l)$ for $0<a'\leq a$ as $2^i-l-1 \ns c=2^i-l-1+c$.  Finally, values in the range $[2^{i+1}+l,3*2^{i}-k-1]$ can be reached by moving to the position $(a',2^i-1,2^i+l)$ for $1\leq a'\leq a$.

To complete this case we need only show that there is no valid move to a position with nimber $3*2^i-k$.  But this is clear: as $k<l\leq 2^{i-1}$ and $a=2^i-k-l$ by \Cref{lemma:addingPowerOfTwo} we cannot reach this nimber by a reduction in $a$ only, and we cannot achieve this value by a reduction in $b$ or $c$ by induction.  Therefore, the claim holds when $a=2^i-k-l$.  To see that the claim holds in the other cases as well, note that from the induction hypothesis we have that incrementing $a$ while reducing $b$ by $1$ fills in the nimber.  Similarly, while $a<2^i-l$ induction also gives us the necessary upper bound.
\end{proof}

While \Cref{thm:same-char} explicitly characterizes nimbers for larger values of $a$, if $\lfloor\log_2(k)\rfloor\neq \lfloor\log_2(l)\rfloor$ then for smaller $a$'s the theorem provides little information. Therefore, we move on to analyzing $\nimber{(a,b,c)}$ in the cases where $\lfloor\log_2(b)\rfloor\neq \lfloor\log_2(c)\rfloor$.

We begin with an instance where we can explicitly determine the values of the Sprague-Grundy function:

\begin{theorem}\label{thm:conjectureone}
Suppose $b=2^i$, $c=(2k+1)2^i+r$, and $a<2^i-r$.  Then $\nimber{(a,b,c)}$ is the $(a+1)^{st}$ gap in $b\ns c$.
\end{theorem}
\begin{proof}
This is done via a nested induction on $a,r,$ and $k$.

For the base, suppose $a=r=k=0$.  Then $\nimber{(a,b,c)}=0$, the $1^{st}$ gap.
Now suppose $0<a<A(b,c)$, $r=k=0$ and assume the claim holds for all smaller values of a.  Then $b=c=2^i$ and we can reach all values less than the $(a+1)^{st}$ gap by either bit arguments or \cref{lemma:gapsarelowerbounds}.  Therefore it suffices to show that there is no move to a position with nimber $a$.  But this is clear: this value cannot be obtained by a reduction in $a$ only (by induction) and any reduction in $b$ to $b'$ (or equivalently $c$) results in a position with nimber at least $b'\ns c+a\geq2^i>a$.

Next, suppose $0<r<2^i$, $0<a<2^i-r$, $k=0$, and the claim holds for all previous values of $a$ and $r$.  Similarly to above, it suffices to show there is no move to a position with the nimber of the $(a+1)^{st}$ gap in $b\ns c$ (in this case the value is just $a+(b\ns c)=a+r<2^i$).  As above, reducing only a, c by more than r, or b at all cannot possibly result in this value (by induction in the first case and the lower bound in the latter two).  Similarly, reducing $c$ by less than $r$ results in some $r'$ in a position with nimber at most $a+r'<a+r$, so there is no valid move to the $(a+1)^{st}$ gap.

Finally, suppose that $a,k,r>0$.  The only additional case to check in this instance are moves that reduce $c$ by more than $2^i$.  However, as any move of this form can only reduce the value of the $(a+1)^{st}$ gap we are done by induction.\end{proof}

Unfortunately, when neither $b$ nor $c$ are a power of two the function's behavior is in general far worse.  While we cannot explicitly characterize the SG function in any more general cases, we can show that when $c$ is sufficiently larger than $b$ order starts to reappear, even for small values of $a$.  We prove this for $b$ odd in the next theorem, but first a lemma:

\begin{lemma} \label{lemma:trailingones}
Let $n>0$ and suppose $b=(2^{i}-1)+n2^i$ and $n2^i\ns c=n2^i+c$ with $i>0$.  Then $A(b,c)\leq 1$.
\end{lemma}

\begin{proof}
Suppose we have a $b$ and $c$ of the desired form and express $c$ uniquelly as $c=m2^i+k$, where $m\geq 0$ and $k<2^i$.  The proof is an induction on $i$ and $k$.

If $i=1$ then $k\in\{0,1\}$.  As $A(b,c)=0$ in the first case for all values of $i$ (taking care of the base cases for each value of $i$) suppose $k=1$.  Then $\nimber{(0,b,c)}=b+c-2$, $\nimber{(0,b-1,c)}=b+c-1$, and $\nimber{(1,b-1,c)}=b+c$ and we are done.

Now suppose the claim holds for all previous values of $i$ and $k$.  Similar to above, we have that $\nimber{(0,b,c)}=b+c-2k$, so it suffices to show that $(1,b,c)\rightarrow x$ for all $x\in [b+c-2k,b+c]$.  If $x\in [b+c-2k,b+c-k]$ then $(0,b,c)\rightarrow x$ by reducing $c$ by some appropriate value less than $k$.  If $x\in [b+c-k,b+c]$ then $(1,b,c)\rightarrow x$ by the I.H. as $\nimber{(1,b,c-k+r)}=b+c-k+r$ for $0\leq r<k$.
\end{proof}
We are now ready to prove Theorem \ref{thm:oddb}\vspace{2mm}.\\
\textbf{Theorem 4. }For $b$ odd, if $c\geq 2^{2\lfloor \log_2 b \rfloor+1}-2^{\lfloor \log_2 b \rfloor+2}-1$ then $\nimber{(1,b,c)}=1+b+c$.

\begin{proof}
We begin by showing this in the case were all of the gaps less than $b$ in $b\ns c$ are consecutive and then showing that the results carry over.

Let $b=2^i+2^j-1$ where $i>j>1$.  From \Cref{lemma:trailingones} we already have that if $c$ does not have a $1$ in its $i$ bit then $\nimber{(1,b,c)}=1+b+c$.  Now consider the sequence of $c$'s where $c$ does have a $1$ in its $i$ bit.  The first such run of $c$'s is $c\in [2^i,2^{i+1}-1]$ and \Cref{thm:same-char} already characterizes these: $\nimber{(1,b,2^i+k)}=k+2^j$ for $k\in [0,2^i-2^j]$ and $b+2^i<\nimber{(1,b,2^i+k)}$ for $k\in [2^i-2^j+1,2^i-1]$.  We use this as the base of an induction showing that for $n=2m+1$ with $m\geq 0$ and $c\in [n2^i,(n+1)2^{i}-1]$ then there are at least $m+1$ values of $c$ for which $b+n2^i<\nimber{(1,b,n2^i+k)}$.  In fact, we claim something slightly stronger: after the $n=1$ case, the $\nimber{(1,b,n2^i+k)}$ 'counts up' along values starting from $(n-2)2^i+1+b$, skipping over at least the $m$ values of $\nimber{(1,b,c)}$ found in the last stage of the induction.

To make this clearer, for each $n>1$ as defined before, as all values in $[(n-3)2^i+b,(n-2)2^i+b]$ can be covered via a reduction in $c$ (from \cref{lemma:trailingones}) it is the case for all appropriate values of $k$ that $b+(n-2)2^i<\nimber{(1,b,n2^i+k)}$.  Now, if there were no reductions in $b$ that could result in a position with nimber $N$ such that $b+(n-2)2^i<N<n2^i$ then as $k$ increases $\nimber{(1,b,n2^i+k)}$ would count up by $1$ for each increase in $k$ but skipping over the $x\geq m$ values found in the last iteration of the induction.  This would happen until the nimber counts up to $n2^i-1$, after which point there are no more gaps in $b\ns c$ less than $(n+1)2^i$.  Further, as all values in the range $[(n-1)2^i+b,n2^i+b]$ can be covered by \Cref{lemma:trailingones}, once the nimbers have counted up to $n2^i-1$ the remaining values will all be greater than $n2^i+b$.  Therefore, as at most $(x-1)$ values were skipped in the last iteration of the induction, leading to $x$ values in the sequence such that $b+n2^i<\nimber{(1,b,n2^i+k)}$, skipping over $x$ values in the count produces at least $x+1$ of the desired values in this iteration.

Note that after at most $2^{i+1}*(2^i-2)$ iterations (in which case $c\geq 2^{2\lfloor \log_2 b \rfloor+1}-2^{\lfloor \log_2 b \rfloor+2}-1$) it is the case that for all greater values of $c$ we have $b+n2^i<\nimber{(1,b,c)}$.  We claim that at this point $\nimber{(1,b,c)}=1+b+c$.  We already had that values in the range $[(n+1)2^i,(n+2)2^i-1]$ for some $n$ achieved the upper bound by \Cref{lemma:trailingones}.  For $c\in[(n)2^i,(n+1)2^i-1]$ for large enough $n$, consider the first value: $c=n2^i$.  In this case the condition that $b+n2^i=b+c<\nimber{(1,b,n2^i)}$ already tells us that $\nimber{(1,b,c)}=1+b+c$.  This in turn inductively tells us that all values of $c$ in this range reach the maximum.

Now, we must deal with the possibility of reductions in $b$ that lead to positions such that $b+(n-2)2^i<\nimber{(a',b,c)}<n2^i$.  To show that such moves cannot lead to issues, consider the first position $c'$ in this iteration of the induction where the nimber differs from the count described in the previous paragraph.  As all values in the range $[(n-1)2^i+b,n2^i+b]$ can still be covered by a reduction in $c$, there are two cases: either $b+(n-2)2^i<\nimber{(1,b,c')}<n2^i$ or $n2^i+b<\nimber{(1,b,c')}$.  In the first case the count is potentially set back by at most $1$ temporarily, but skips the value of $\nimber{(1,b,c')}$ later in the count for no net change.  Similarly, in the latter case although the count can potentially be set back by $1$ for its entire duration, $\nimber{(1,b,c')}$ becomes one of the $m$ values needed for the induction to work.  As this is the case whenever a position differs from what is predicted by the count no problems arise.

Finally, suppose that not all gaps of $b \ns c$ are consecutive.  Then $b=2^j-1+(2n+1)2^{i}$ for some $i>j+1>0$ and note that applying the procedure from before on $2^i+2^j-1$ shows that for large enough $c$ $\nimber{(1,2^i+2^j-1,c)}=1+2^i+2^j-1+c$.  As none of the arguments necessary to prove this are effected by the addition of leading $1$'s in $c$, this procedure can be applied inductively to each sub-component of $b$ (based on the number of leading ones in $b$) to show the result in general.
\end{proof}

For $b$ even, while a similar analysis can provide periodicity results in the $a=1$ case, doing so is far more dependent on the initial conditions of the induction.  This is due to the following lemma, which ensures that $\nimber{(1,b,c)}\neq 1+b+c$ for values when $c$ is also even and $b \ns c \neq b+c$, and thus complicates the recursive structure of $(1,b,c)$.

\begin{lemma} \label{lemma:evens}
If $b$ and $c$ are both even, then $A(b,c)\neq1$
\end{lemma}

\begin{proof}
Suppose $b=2^i+2m$ and $c=2^{i+r}+2n$.  The proof is again via nested induction:

From Theorems 2 and 6, if any of $m$, $n$ or $r$ are $0$ then either $A(b,c)=0$ or $A(b,c)\geq 2$ as desired.  This covers the base case for each part of the induction.

Now suppose $b=2^i+2m$ and $c=2^{i+r}+2n$ where $m,n,r>0$ and the claim holds for all previous values $m,n,r$.  If $b\ns c=b+c$ then we are done.  If not, there are two cases: either the bit representation of $b$ and $c$ intersect only in their rightmost filled bit or not.

If we are in the first case, let $x$ be the index of the rightmost filled bit of $b$ and $c$.  Then $b\ns c=b+c-2^{x+1}$ and $(1,b,c)\nrightarrow b+c-2^x-1$.  This is because the trivial upper and lower bounds give that this value can only possibly be achieved by a reduction in $b$ or $c$ by either $2^x-1$ or $2^x-2$.  However, in the first case \cref{lemma:trailingones} gives us that the resulting nimber will be too large, and in the latter case the IH gives the resulting nimber will be too small.  Therefore, in this case $\nimber{(1,b,c)}\leq b+c-2^x-1$.

Now suppose we are in the second case.  Consider how $(1,b,c)\rightarrow b+c$ and $(1,b,c)\rightarrow b+c-1$.  To reach $b+c$, it must be the case that either $A(b-1,c)$ or $A(b,c-1)=1$, so WLOG assume $A(b-1,c)=1$.  Then as $b$ and $c$ overlap somewhere other than their rightmost filled bit it's the case that both $A(b-2,c)$ and $A(b,c-2)\neq0$.  Therefore, by the IH $(1,b,c)$ cannot reach $b+c-1$ by a reduction in $b$ or $c$ by two.  Therefore, unless $\nimber{(1,b,c-1)}=b+c-1$ the claim holds.  However, under these circumstances in order for $A(b-1,c)=1$ it must be the case that $A(b-1,c-1)=1$.  But then it's impossible for $\nimber{(1,b,c-1)}=b+c-1$ and the proof is complete.
\end{proof}

Therefore, while we can prove periodicity results for $b$ even and $a=1$ in several cases, there are enough exceptions to the general rule that we cannot do so in general.  However, for $a=2$ a similar analysis to Theorem $\ref{thm:oddb}$ should show that $\nimber{(2,b,c)}=1+b+c$ for all large $c$.

\section{Discussion}
\subsection{Further Directions with Auxiliary Nim}
To recap, at this point we have characterized the Sprague-Grundy function of $(a,b,c)$ whenever: (1) $a$ is sufficiently large; (2) $\lfloor log_2(b)\rfloor=
\lfloor log_2(c)\rfloor$, or (3) $c>>b$. In some cases we have also extended these results to general auxiliary-nim games.

One potential line of further work is doing a more detailed analysis of the remaining cases: can we give a closed form expression for $\nimber{(a,b,c)}$?
\begin{question} Determine a non-recursive description of the behaviour of $\nimber(a,b,c)$.
\end{question}
Figure \ref{fig:fig} suggests that a closed-form solution, at least a simple one, is unlikely to emerge.\\[0.01in]
\par We have also not fully analyzed how the results regarding the $c>>b$ cases might generalize to the general Auxiliary Nim.
\begin{question}Characterize $\nimber(x_1,x_2,\cdots,x_n)$ when $x_n$ is ``sufficiently large''.\end{question}

Perhaps more interestingly, however, more general ``auxiliary'' games could be analyzed.  What can we say about the game $(*k)\ws{A}$, where $A$ is an arbitrary impartial combinatorial game?

\begin{question}Characterize the games $A$ where $\exists k_0\in \mathbb{N}$  such that $\forall k>k_0$, $\nimber{((*k)\ws A)}=k+depth(A)$.
\end{question}

We already know that Nim has this property.  Do more exotic games?\\[0.05in]

\par Using the notation presented in \cite{tetrishypergraphs} we note that $n$ heap auxiliary-nim is the game $NIM_\mathcal{H}$ where $\mathcal{H}=\{\{1\},...,\{n\}, \{1,2\},\{1,3\},...,\{1,n\}\}$.  Here, the game $NIM_\mathcal{H}$ is played on $|V(\mathcal{H})|$ heaps were a valid move is selecting a hyperedge in $\mathcal{H}$ and making reductions in all non-empty heaps within that edge.  Are there more general hypergraphs $\mathcal{H}$ where $NIM_{\mathcal{H}}$ behaves similarly to Auxiliary Nim?
\begin{question} Do results presented here extend to more general hypergraph games?\end{question}

\subsection{Periodicity}

We do know that not all games $A$ satisfy the property mentioned in Question $3$. For example, consider games of the following form:

\begin{definition}
A general subtraction game is a sequence of games $G_n$ such that the set of positions that $G_n$ can move to is $\{G_m \mid m \in g(n)\}$ where $g: \mathbb{N} \rightarrow 2^{\mathbb{N}}$ is such that $\forall \: n \in \mathbb{N}$, $g(n) \subseteq [n-1]$. We call $g$ the function associated with $G_n$
\end{definition}

\begin{definition}
A finite fixed set subtraction game is a subtraction game $G_n$ such that there exists a set $S \subseteq N$ for some $N\in \mathbb{N}$ such that the function $g$ associated with $G_n$ satisfies $g(n) = \{n-x \mid x \leq n \wedge x \in S\}$. We call $S$ the set of $G_n$.
\end{definition}

It is not hard to prove that the Sprague-Grundy values for $*k\ws{G_n}$ is periodic with respect to $n$ if $G_n$ is a finite subtraction game, although the upper bound on the length of the period is exponential. Note that periodicity immediately tells us that the property mentioned in Question $3$ cannot hold.

\begin{theorem}
If $G_n$ is a finite subtraction game, then the Sprague-Grundy function of $G_n \ws *k$ is periodic for any $k \in \mathbb{N}$.
\end{theorem}

\begin{proof}
Let $G_n$ be a finite subtraction game with set $S$, and let $m = \text{max}(S)+1$. Since any position in $G_n \ws *k$ has at most $m$ choices for which move to make in the left game (note that $m$ is larger now because we include the possibility of not moving in the left game), and at most $k+1$ choices for which move to make in the right game, the total number of moves possible from $G_n \ws *k$ is at most $(k+1)m$, and thust $\nimber{(G_n \ws *k)} \leq (k+1)m$ (so the nimbers are bounded).

Note also that the nimber of $G_n \ws *k$ is completely determined by the nimbers of $G_{n-x} \ws *(k-y)$ where $0 < x \leq m$ and $0 \leq y \leq k$. Note that we need not consider $x = 0$, because in fact the nimbers for the positions of this form where $x = 0$ are completely determined by the rest. That is, the $\nimber{(G_n \ws *0)}$ is completely determined by $\{\nimber{(G_{n-x} \ws *0)}\}$, and thus $\nimber{(G_n \ws *1)}$ is completely determined by $\{\nimber{(G_{n-x} \ws *0)}\} \cup \{\nimber{(G_{n-x} \ws *1)}\} $, and so on.

Thus, if we have that for some $a,b \in \mathbb{N}$, and for every $0 < x \leq m$ and $0 \leq y \leq k$, $\nimber{(G_{a-x} \ws *(k-y))} = \nimber{(G_{b-x} \ws *(k-y)))}$, then we must also have that for every $0 \leq y \leq k$, $\nimber{(G_a \ws *(k-y))} = \nimber{(G_b \ws *(k-y))}$. Thus, if such an $a$ and $b$ exist with $a \neq b$ we have, by induction, that $G_n \ws *k$ is periodic with period at most $|b-a|$.

To see that such an $a$ and $b$ must exist, we simply note that since the nimbers are bounded by $(k+1)m$, and the number of choices for $x$ and $y$ is only $(k+1)m$, there are only $((k+1)m)^{(k+1)m}$ possibilities for the nimbers of the positions for the form $G_{n-x} \ws *(k-y)$, so by PHP, there must exist $0 \leq a < b < m + ((k+1)m)^{(k+1)m}$ such that for every $0 < x \leq m$ and $0 \leq y \leq k$, $\nimber{(G_{a-x} \ws *(k-y))} = \nimber{(G_{b-x} \ws *(k-y))}$, and thus, by the above observations, $G_n \ws *k$ is periodic.
\end{proof}

It's not hard to construct artificial sequences of games $A_n$ such that $A_n$ is periodic, but $*1\ws A_n$ is not. However, it appears as though if the sequence is constructed with certain structural regularities, such as the case of finite subtraction games, periodicity seems to be preserved.  Therefore, we have another interesting question at hand.

\begin{question} For which sequences of games $A_n$ is $\nimber{(*k\ws{A_n})}$ periodic with respect to $n$ for any $k\in \mathbb{N}$?
\end{question}
\par For instance, consider the game $GRAPH_G$ played on a simple graph $G$: on each turn, the players select a vertex, and remove a positive integer many edges incident on that vertex. Terminal positions are edgeless graphs. When this game is played on a path graph, it is isomorphic to a game of $Kayles$ \cite{siegel}. $KAYLES_n$ (or $GRAPH_{P_n}$ where $P_n$ is a path of edge-length $n$) is known to be periodic with a period of $12$. The proof of this fact is data-driven: there exists a threshold value of $N$ such that when $KAYLES_n$ is verified computationally to be periodic up to the threshold value, then we can deduce that it will remain periodic forever. This threshold argument works for a large class of games.
\begin{definition}
An octal game is a game played with tokens divided into heaps, where valid moves are one of the following:
\begin{itemize}
    \item Remove some (possibly all) of the tokens in one heap
    \item Remove some (not all) of the tokens in a heap, and divide the rest into two non-empty heaps.
\end{itemize}
\end{definition}
Observe that normal single-heap Nim is an octal-game, but not periodic. The following theorem formalizes the threshold argument for most octal games. Call $G_n$ (starting configuration is single heap with $n$ tokens) a bounded octal game if the number of tokens that can be removed from any single heap is bounded.
\begin{theorem}\label{periodicity}
Let $G_n$ be a bounded octal game with bound $k\in\mathbb{N}$. Suppose that $\exists\, n_0,p\geq 1$ such that $\nimber(G_n)=\nimber(G_n+p)$ for all $n$ satisfying $n_0\leq n\leq 2n_0+p+k$. Then, $G_n$ is periodic.
\end{theorem}
The proof follows by a simple induction on $n$. For a proof and a more extensive survey, see $\cite{siegel}$. A prominent conjecture in combinatorial game theory, initially proposed by John Conway, is the following:
\begin{conjecture}
All bounded octal games are periodic.
\end{conjecture}
The conjecture is convincing, but it offers no upper bound on the period and computational verification on a large scale is mostly intractable.
\par Disappointingly, other than through Theorem \ref{periodicity} and computational search, we don't have a way to prove that a sequence of games will be periodic, even given that a sequence with almost identical structure is periodic. We believe however that this is a promising direction. Consider the following game:
\begin{definition}
$STARKAYLES_{k,n}$ is the game $GRAPH_G$, where $G$ is obtained by starting with a star graph on $k$ vertices, and then extending one of the branches to be a path of edge-length $n$.
\end{definition}
\par Observe that $STARKAYLES_{1_n}$ is the same as $KAYLES_n$. We have computationally verified for small values of $k$ that $STARKAYLES_{k,n}$ is periodic, with period a multiple of $12$. We conjecture that this generalizes, since the fixed star should not intuitively have a structural effect on the asymptotic behaviour of the sequence.
\begin{conjecture}
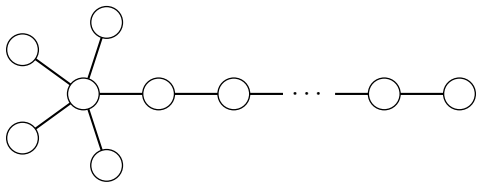
\begin{figure}
\centering
\begin{tikzpicture}
\renewcommand*{\VertexInterMinSize}{12pt}
\SetVertexNoLabel
\Vertex{Z} \Vertices{circle}{A,B,C,D,E} \EA(A){F} {\GraphInit[vstyle=Empty]\SetVertexLabel \EA[L=$\cdots$](F){G}} \EA(G){H} \EA(H){I}

\Edge(Z)(A) \Edge(Z)(B) \Edge(Z)(C) \Edge(Z)(D) \Edge(Z)(E) \Edge(A)(F) \Edge(F)(G) \Edge(G)(H) \Edge(H)(I)
\end{tikzpicture}
\caption{The game $STARKAYLES_{5,n}$. A valid move is picking a vertex, and removing a positive number of edges from it. We conjecture that all games of this form will be periodic.}
\end{figure}
For all $k$, $STARKAYLES_{k,n}$ is periodic, with period a multiple of $12$.
\end{conjecture}
To move beyond computational verification, we suggest the following direction of research:
\begin{question}
Can we prove that $STARKAYLES_{2,n}$ is periodic, without relying on Theorem \ref{periodicity}, and only on the fact that $KAYLES_n$ is periodic?
\end{question}
Of course, there should not be anything special about starting with a star as opposed to any other fixed graph, and extending a path of length $n$ from a vertex. However,  $STARKAYLES_{2,n}$ seems to be the simplest extension to $KAYLES$ that also preserves periodicity.
\par The operation $\ws(*k)$ cannot model attaching a fixed graph to a vertex in $KAYLES_n$; however, it's similar. We also conjecture the following:
\begin{conjecture}
$KAYLES_n\ws(*1)$ is periodic.
\end{conjecture}
This conjecture is virtually impossible to computationally verify, since computing nimbers involve looking at roughly $P(n)$ (partition number of $n$) many games (which is exponential in $n$), as the $\ws(*1)$ prevents us from calculating the nimber of a disjoint union of $KAYLES$ games by simply XORing the nimbers. We hope that techniques that can address Question $6$ can generalize to prove Conjecture $3$.

\end{document}